\documentclass[a4paper]{article}
\usepackage{mathtext}
\usepackage{amssymb}
\usepackage{amsfonts}
\usepackage{amsthm}
\usepackage{amsmath}
\usepackage{latexsym}
\usepackage{geometry}
\usepackage{graphicx}
\usepackage{mathdots}
\usepackage{amsmath,amscd}
\usepackage{xypic}

\newtheorem{theorem}{Theorem}

\newtheorem{corollary}[theorem]{Corollary}

\newtheorem{definition}[theorem]{Definition}

\newtheorem{lemma}[theorem]{Lemma}

\newtheorem{proposition}[theorem]{Proposition}
\newtheorem{remark}[theorem]{Remark}

\begin{document}

\title{Limits of Quotients of Polynomial Functions in Three Variables}
\author{Juan D. V\'{e}lez,~~ Juan P. Hern\'{a}ndez,~~ Carlos A. Cadavid.}
\date{}
\maketitle

\begin{abstract}
A method for computing limits of quotients of real analytic functions in two
variables was developed in \cite{CSV}. In this
article we generalize the results obtained in that paper to the case of
quotients $q=f(x,y,z)/g(x,y,z)$ of polynomial functions in three variables
with rational coefficients. The main idea consists in examining the behavior
of the function $q$ along certain real variety $X(q)$ (the \emph{%
discriminant variety associated to} $q)$. The original problem is then
solved by reducing to the case of functions of two variables. The inductive
step is provided by the key fact that any algebraic curve is birationally
equivalent to a plane curve. Our main result is summarized in Theorem \ref%
{tl2}.

In Section 4 we describe an effective method for computing such limits. We
provide a high level description of an algorithm that generalizes the one
developed in \cite{CSV}, now available in \textit{Maple} as the \texttt{%
limit/multi} command.
\end{abstract}

\section{Introduction}

Algorithms for computing limits of functions in one variable are studied in 
\cite{G3}. Similar algorithms have been developed in \cite{G1} and \cite{G2}%
. Computational methods dealings with classical objects, like power series
rings and algebraic curves, have been developed by several authors during
the last two decades, \cite{AMR} and \cite{SSH}. A symbolic computation
algorithm for computing local parametrization of analytic branches and real
analytic branches of a curve in $n$-dimensional space is presented in \cite%
{ANMR}.

In \cite{CSV} V\'{e}lez, Cadavid and Molina developed a method for analyzing
the existence of limits $\lim_{(x,y)\rightarrow (a,b)}q(x,y),$ where $q(x,y)$
is a quotient of two real analytic functions $f$ and $g,$ under the
hypothesis that $(a,b)$ is an isolated zero of $g$. In the case where $f$
and $g$ are polynomial functions with rational coefficients, the techniques
developed in that article provide an algorithm for the computation of such
limits, now available in \textit{Maple} as the \texttt{limit/multi} command 
\cite{link}.

An alternative method for computing limits of quotients of functions in
several variables has been recently developed in \cite{Xiao}. Their approach
is completely different from ours, relaying on Wu's algorithm as the main
tool.

In this article we generalize the methods presented in \cite{CSV} to the
case of quotients of polynomials in three variables, under the same
assumption that $g$ is a function with an isolated zero at the point $(a,b)$%
. The main idea consists in reducing the problem of determining the
existence of limits of the form 
\begin{equation}
\lim_{(x,y,z)\rightarrow (a,b,c)}f(x,y,z)/g(x,y,z)  \label{xxx}
\end{equation}%
to the problem of determining the limit along some real variety $X(q)$
associated to $q$ (the \emph{discriminant variety of }$q$). In order to
achieve this one needs to study the topology of the irreducible components
of the singular locus of $X(q)$. The original problem is then solved by
reducing to the case of functions of two variables. The inductive step is
provided by the key fact that any algebraic curve is birationally equivalent
to a plane curve. Our main result is summarized in Theorem \ref{tl2}. In
Section 4 we provide a high level description of a potential algorithm
capable of determining the existence of (\ref{xxx}), and if the limit
exists, it would be able to determine its value. \ Any of the Groebner Basis
packages available may serve as a computational engine to implement such an
algorithm. In Section 5 we present two examples that illustrate some the
computation that would be needed in a typical problem of determining and
computing a limit of this sort.

\section{Preliminaries}

\subsection{Dimension of algebraic sets and its singular locus}

\label{crzv}

In this article we consider complex affine varieties defined by polynomials
with real coefficients. If $I$ is an ideal in the polynomial ring $S=\mathbb{%
\mathbb{R}}[x_{1},\dots ,x_{n}]$, by $X=V(I)$ we will denote the \emph{%
complex} \emph{affine variety} defined by $I,$ i.e., the common zeros of $I$
in $\mathbb{C}^{n}$. The \emph{dimension }of $X$ is the Krull dimension of
the ring $\mathbb{C\otimes }_{\mathbb{R}}S/I$. Since $S/I\subset \mathbb{%
C\otimes }_{\mathbb{R}}S/I$ is a faithfully flat extension of rings, the
dimension of $X$ coincides with the dimension of $\mathbb{R}[x_{1},\dots
,x_{n}]/I$, the \emph{real affine ring of }$X.$ It is well known that if $X$
is irreducible, defined by some prime ideal $P\subset S$, then the dimension
of the domain $R=S/P$ coincides with the transcendence degree of the field
extension $\mathbb{R}\subset L$ (denoted by trdeg$_{K}L),$ where $L$ denotes
the fraction field of $R$.

We recall the definition of the \emph{singular locus} of a equidimensional
affine variety.

\begin{definition}
\label{sing}Let $Y\subset \mathbb{C}^{n}$ be an affine variety, and let $R=%
\mathbb{C}[x_{1},\ldots ,x_{n}]/I(Y)$ be its ring of coordinates. Suppose
that $R$ is equidimensional of dimension $r$ (i.e., $ht(P)=r,$ for all the
minimal primes $P$ containing $I(Y)$). Let's choose arbitrary generators $%
f_{1},\dots ,f_{k}$ for $I(Y).$ The singular locus of $Y$, denoted by 
\textrm{Sing}($Y$), is the closed subvariety of $Y$ defined by the ideal $%
J=I(Y)+$ the ideal of all $(n-r)\times (n-r)$ minors of the Jacobian matrix.
\end{definition}

\begin{remark}
\label{rsp} \ \ \ \ \ \ \ 

\begin{enumerate}
\item The above criterion to determine \textrm{Sing}($Y$) does not depend on
the generators one chooses for $I(Y)$.

\item The singular locus \textrm{Sing}($Y$) is a proper closed subvariety of 
$Y$, defined by those points $p\in Y$ for which the rank of the Jacobian
matrix $[(\partial f_{i}/\partial x_{j})(p)]$ is less that $n-r$.

\item $\dim ($\textrm{Sing}$(Y))<\dim (Y)$
\end{enumerate}
\end{remark}

(See \cite{Eisenbud}, Section 16.5 and \cite{Hartshorne}, Chapter I, Section
5).

\bigskip

We will mainly focus in the following simple case: Suppose that $X\subset 
\mathbb{C}^{3}$ is an affine variety of dimension $2$ defined by a prime
ideal $P\subset \mathbb{R}[x,y,z].$ In this case $P\subset \mathbb{R}[x,y,z]$
must be a prime ideal of height $1,$ and so it has to be principal, i.e., $%
P=(h),$ where $h\in \mathbb{R}[x,y,z]$ is some real irreducible polynomial.
Therefore, $X=V(h)$. In this case \textrm{Sing}($X$) is the complex affine
variety defined by the ideal $I_{S}=(h,\partial h/\partial x,\partial
h/\partial y,\partial h/\partial z)\subset \mathbb{R}[x,y,z]$.

\subsection{The discriminant variety\label{discrimina}}

The existence of $\lim_{(x,y,z)\rightarrow (a,b,c)}f(x,y,z)/g(x,y,z)$ does
not depend on the particular choice of local coordinates. Hence, after an
appropriate translation we may always assume that $p=(a,b,c)$ is the origin,
here denoted by $O$. Our objective is to compute 
\begin{equation}
\lim_{(x,y,z)\rightarrow (0,0,0)}f(x,y,z)/g(x,y,z),  \label{eq1}
\end{equation}%
where $f(x,y,z)$ and $g(x,y,z)$ are rational polynomial functions, and where 
$g$ has an isolated zero at $O$. If $q(x,y,z)=f(x,y,z)/g(x,y,z),$~we define
the\emph{\ discriminant variety} $X(q)$ associated to $q$ as the variety
defined by the $2\times 2$ minors of the matrix 
\begin{equation}
A=\left[ 
\begin{array}{ccc}
x & y & z \\ 
\partial q/\partial x & \partial q/\partial y & \partial q/\partial z%
\end{array}%
\right] .  \label{matriz}
\end{equation}%
Strictly speaking, the $2\times 2$ minors of $A$, $x_{i}\partial q/\partial
x_{j}-x_{j}\partial q/\partial x_{i}$, are not necessarily polynomial
functions. However, these minors can be written as 
\begin{equation*}
x_{i}\partial q/\partial x_{j}-x_{j}\partial q/\partial x_{i}=\frac{%
x_{i}(g\partial f/\partial x_{j}-f\partial g/\partial x_{j})}{g^{2}}-\frac{%
x_{j}(g\partial f/\partial x_{i}-f\partial g/\partial x_{i})}{g^{2}},
\end{equation*}%
and therefore, if we let%
\begin{equation*}
f_{x_{i},x_{j}}=x_{i}(g\partial f/\partial x_{j}-f\partial g/\partial
x_{j})-x_{j}(g\partial f/\partial x_{i}-f\partial g/\partial x_{i}),
\end{equation*}%
then the variety $X(q)$ can be defined as the zeros of the ideal $%
J=(f_{x,y},f_{x,z},f_{y,z}).$

The following proposition states that in order to determine the existence of
the limit (\ref{eq1}) it suffices to analyze the behavior of the function $%
q(x,y,z)$ along the discriminant variety $X(q)$.

\begin{proposition}
\label{pl1}The limit $\lim_{(x,y,z)\rightarrow 0}q(x,y,z)$ exists, and
equals $L\in \mathbb{R}$, if and only if for every $\epsilon >0$ there is $%
\delta >0$ such that for every $(x,y,z)\in X(q)$ with $0<|(x,y,z)|<\delta $
the inequality $|q(x,y,z)-L|<\epsilon $ holds.
\end{proposition}

\begin{proof}
The method of Lagrange multipliers applied to the function $q(x,y,z)$ with
the constraint $x^{2}+y^{2}+z^{2}=r^{2}$, $r>0$ guarantees that if $%
C_{r}(0)=\{(x,y,z)\in \mathbb{R}^{3}~:~x^{2}+y^{2}+z^{2}=r^{2}\}$ then the
extreme values of $q(x,y,z)$ on $C_{r}(0)$ are taken at those points $%
p=(a,b,c)\in C_{r}(0)$ for which $(\partial q/\partial x(p),\partial
q/\partial y(p),\partial q/\partial z(p))=\lambda (a,b,c)$, i.e., at those
points in $X(q)$.\newline
Suppose that given $\epsilon >0$ there is $\delta >0$ such that for every $%
(x,y,z)\in X(q)\cap D_{\delta }^{\ast }$ the inequality $|q(x,y,z)-L|<%
\epsilon $ holds, where~ $D_{\delta }^{\ast }=\{(x,y,z)\in \mathbb{R}%
^{3}~:~0<\sqrt{x^{2}+y^{2}+z^{2}}<\delta \}$. Let $(x,y,z)\in D_{\delta
}^{\ast }$ and $r=\sqrt{x^{2}+y^{2}+z^{2}}$. If $t(r),s(r)\in C_{r}(0)$ are
respectively the maximum and minimum values of $q(x,y,z)$, subject to $%
C_{r}(0)$, then 
\begin{equation*}
q(s(r))-L\leq q(x,y,z)-L\leq q(t(r))-L.
\end{equation*}%
As $t(r)$ and $s(r)\in X(q)\cap C_{r}(0)\subset X(q)\cap D_{\delta }^{\ast }$%
, one sees that $-\epsilon <q(s(r))-L$, and henceforth $q(t(r))-L<\epsilon $%
. Thus, $|q(x,y,z)-L|<\epsilon $.\newline
The reciprocal is obvious.
\end{proof}

\subsection{Birational equivalence of curves}

We intend to reduce the problem of determining the existence of the limit (%
\ref{xxx}) to a problem in fewer variables. In order to achieve this we will
use the fact that any algebraic curve is birationally equivalent to a plane
curve. This result follows from the following standard result:

\begin{proposition}[existence of primitive elements]
\label{t1}Let $K$ be a field of characteristic zero and let $L$ be a finite
algebraic extension of $K$. Then there is $z\in L$ such that $L=K(z)$ (\cite%
{Fulton}, Page 75).
\end{proposition}

This immediately implies the following corollary:

\begin{corollary}
\label{c1} Let $X$ be an irreducible algebraic curve over a field $k$ of
characteristic zero, and let $K$ be the quotient field of the ring of
coordinates of $X$. Then for any $x\in K-k$ which is not algebraic over $k$, 
$K$ is algebraic over $k(x)$, and there is an element $y\in K$ such that $%
K=k(x,y) $.
\end{corollary}

The next theorem is a well known fact. Notwithstanding, we give a proof
since we will need the explicit construction of the isomorphism denoted by $%
\mu$ in the following theorem.

\begin{theorem}
\label{t2}\label{rnew copy(1)} Let $X$ be an irreducible space curve $X$ in $%
\mathbb{C}^{3}$ defined by polynomials with real coefficients, and such that
the origin of $\mathbb{C}^{3}$ is a point of $X$. Then there exists an
irreducible affine plane curve $Y\subset \mathbb{C}^{2}$ and a field
isomorphism $\varphi :K(Y)\rightarrow K(X)$ so that $X$ is birationally
equivalent to $Y$. After removing a finite set of points $Z\subset Y$, if $%
Y_{0}=Y\setminus Z$, then there is a morphism $\mu :X\rightarrow Y_{0}$ such
that $\mu $ restricted to $X_{0}=\mu ^{-1}(Y_{0})$ is an isomorphism onto $%
Y_{0}.$ Both $\mu $ and its inverse can be explicitly constructed.
\end{theorem}

\begin{proof}
Suppose that $X=V(P),$ where $P\subset \mathbb{R}[X,Y,Z]$ is a prime ideal.
Since $X$ is an irreducible algebraic curve $\text{dim}(X)=\text{dim}(%
\mathbb{R}[X,Y,Z]/P)=1$. Denote by $\mathbb{R}(x,y,z)$ the fraction field of 
$\mathbb{R}[X,Y,Z]/P$. Recall that $\text{dim}(X)=\text{trdeg}_{\mathbb{R}}%
\mathbb{R}(x,y,z)$.\newline
For dimensional reasons some of the variables $x,y$ or $z$ has
to be transcendental over $\mathbb{R}$. Suppose without loss of generality
that $x$ is transcendental over $\mathbb{R}$. Corollary \ref{c1} implies that
$\mathbb{R}(x)\subset \mathbb{R}(x,y,z)$ is an algebraic extension. By
Proposition \ref{t1}, one can always find $u=y+\lambda z,$ for some $\lambda
\in \mathbb{R}(x),$ such that $\mathbb{R}(x,y,z)=\mathbb{R}(x,u)$. Moreover,
since this is true for almost all $\lambda $, this element can be taken to
be any real constant, except for finitely many choices. Define $\varphi :~%
\mathbb{R}[S,T]\rightarrow \mathbb{R}[x,u]\subset \mathbb{R}(x,y,z)$ as the $%
\mathbb{R}$-algebra homomorphism that sends $S\rightarrow x$ and $%
T\rightarrow u$. Clearly $\varphi $ is surjective, and therefore, if $J=\text{%
ker}(\varphi ),$ there is an isomorphism of $\mathbb{R}$-algebras $\varphi :%
\mathbb{R}[S,T]/J\overset{\sim }{\rightarrow }\mathbb{R}[x,u].$
Consequently, $J\in \mathbb{R}[S,T]$ is a prime ideal. Denote $V(J)$ by $Y$.
The last isomorphism induces a field isomorphism $\varphi :\mathbb{R}%
(Y)\cong \mathbb{R}(x,u)\rightarrow \mathbb{R}(x,y,z)$ defined as $\varphi
(x)=x,$ $\varphi (u)=y+\lambda z$. Therefore, $\text{dim}(Y)=\text{dim}%
(X)=1$. Hence, $Y=V(J)$ is an irreducible algebraic plane curve which is
birationally equivalent to $X$.\newline
The morphism $\varphi :\mathbb{R}(Y)\rightarrow \mathbb{R}(X)$ induces a
morphism of varieties $\mu :X\rightarrow Y$ given by $\mu
(a,b,c)=(a,b+\lambda c)$.

Notice that since $(0,0,0)\in X,$ then, obviously, $(0,0)\in Y$. Since $Y$ is
an irreducible plane curve, $J$ must be a
height one prime ideal. Thus, $J=(h),$ for some $h(X,U)\in \mathbb{R}[X,U]$.%
\newline
We can assume that the polynomial $h(a,U)$ obtained by replacing the
variable $X$ by $a\in \mathbb{C}$ is not identically zero: If $h(a,U)=0$ we
would have $h(X,U)=(X-a)^{m}t(X,U),$ with $t(a,U)\neq 0$. But $(0,0)\in Y$
implies $a=0$, and henceforth $h(X,U)=X^{m}t(X,U)$. Thus, $h(X,U)=X$ or $%
h(X,U)=t(X,U),$ since $Y$ is irreducible. Finally, we note that $h(X,U)=X$
contradicts the fact that $x$ is transcendental over $\mathbb{R}$.

On the other hand, since $x$ is transcendental over $\mathbb{R}$, by
Corollary \ref{c1}, the extension $\mathbb{R}(x)\subset \mathbb{R}(x)(u)$ is
algebraic. Therefore, since $y\in \mathbb{R}(x,u)$ one can write $y$ as: 
\begin{equation*}
y=\frac{a_{0}(x)}{b_{0}(x)}+\frac{a_{1}(x)}{b_{1}(x)}u+\cdots +\frac{a_{r}(x)%
}{b_{r}(x)}u^{r},
\end{equation*}%
where $r$ is smaller than the degree of the field extension $[\mathbb{R}%
(x)(u):\mathbb{R}(x)]$. Taking $b(x)=b_{0}(x)\cdots b_{r}(x)$ we can rewrite
the last equation as 
\begin{equation}
y=\frac{c_{0}(x)+c_{1}(x)u+\cdots +c_{r}(x)u^{r}}{b(x)},  \label{zzz}
\end{equation}%
for certain $c_{i}(x)$. Therefore, we have $y=f_{1}(x,u)/g_{1}(x)$ and $z=f_{2}(x,u)/g_{2}(x)$. Consider $Z=\{(a,b)\in
Y~:~g_{1}(a)=0~~\text{or}~~g_{2}(a)=0\}$, which is a Zariski closed subset of 
$Y$.\newline
Let us see that $Z$ is a finite set. Indeed, the polynomials $g_{1}$ and $%
g_{2}$ have finitely many roots. Therefore, if $a_{1},\dots ,a_{k}\in 
\mathbb{C}$ are these roots, for each $a_{i}$, $(a_{i},b)\in Y$ if and only
if $h(a_{i},b)=0$, where $Y=V(J)$ with $J=(h)$. Notice that the polynomial $%
t(U)=h(a_{i},U)\in \mathbb{C}[U]$ has finitely many roots. Hence, there
are only finitely many elements $(a_{i},b)$ with $g_{1}(a_{i})=0$ or $%
g_{2}(a_{i})=0$, and such that $h(a_{i},b)=0$. Thus, we conclude that $Z$ is
finite.\newline
Consider the open subset $Y_{0}=Y\setminus Z$ of $Y$. Let $X_{0}=\mu
^{-1}(Y_{0})$.
Define $\tau :Y_{0}\rightarrow X_{0}$ as $\tau (d,e)=(d,\frac{f_{1}(d,e)}{%
g_{1}(d)},\frac{f_{2}(d,e)}{g_{2}(d)})$.
This last morphism induces an $\mathbb{R}$-algebra homomorphism $\psi :%
\mathbb{R}(X)\rightarrow O_{Y}(Y_{0})$ given by $\psi (x)=s,~\psi
(y)=f_{1}(s,t)/g_{1}(s)$, and $\psi (z)=f_{2}(s,t)/g_{2}(s)$. Clearly, 
\begin{equation}
\varphi \circ \psi (x)=x,~~\varphi \circ \psi (y)=\frac{f_{1}(x,u)}{g_{1}(x)}%
=y,~~\varphi \circ \psi (z)=\frac{f_{2}(x,u)}{g_{2}(x)}=z.  \label{eq2}
\end{equation}%
Therefore, $\varphi \circ \psi =Id_{\mathbb{R}(X)},$ and consequently $%
\varphi \circ \psi |_{X_{0}}:O_{X}|_{X_{0}}\rightarrow O_{Y}|_{Y_{0}}$ is
the identity. On the other hand, $\psi \circ \varphi (s)=\psi (x)=s$ and $%
\psi \circ \varphi (t)=\psi (u)$. By (\ref{eq2}) we have $\varphi \circ \psi
(u)=u$ and $\varphi (t)=u$, which implies that $t=\psi (u),$ since $\varphi $
is injective. Hence, $\psi \circ \varphi (t)=t$ and therefore  $\psi
\circ \varphi |_{Y_{0}}:O_{Y}|_{Y_{0}}\rightarrow O_{X}|_{X_{0}}$ is the
identity. Hence, $\psi :O_{X}|_{X_{0}}\rightarrow O_{Y}|_{Y_{0}}$ is the
inverse of the morphism $\varphi :O_{Y}|_{Y_{0}}\rightarrow O_{X}|_{X_{0}}.$ Thus, the homomorphism $\tau :Y_{0}\rightarrow X_{0}$ induced by $\psi $
is the inverse of $\mu :X_{0}\rightarrow Y_{0}$.\newline
Finally, it is clear that the morphism $\mu :X_{0}\rightarrow Y_{0}$ sends
the real part of $X_{0}$ into the real part of $Y_{0}$, and since $\mu
^{-1}=\tau :Y_{0}\rightarrow X_{0}$ is determined by the polynomials $%
f_{1},f_{2},g_{1}$ and $g_{2}$, which are all real polynomials, then $\mu
^{-1}=\tau $ also sends the real part of $Y_{0}$ into the real part of $X_{0}
$.
\end{proof}

\begin{remark}
\label{rnew2} $X_{0}$ is obtained from $X$ by removing finitely many points. 
\begin{proof}

In fact, a point $(a,b,c)\in X$ does not belong to $X_{0}$ iff $\mu
(a,b,c)=(a,b+\lambda c)\notin Y_{0}$, i.e., iff $(a,b+\lambda c)\in Z$. But $%
Z$ is finite, and therefore there are only finitely many choices for $a$ and 
$(b+\lambda c)$ such that $(a,b+\lambda c)\notin Y_{0}.$ Fix any values for $a$
and for $\eta $=$b+\lambda c$. If $f_{1}(x,y,z),\dots
,f_{k}(x,y,z)$ are generators for $P$ then, clearly, $f_{i}(a,\eta -\lambda
c,c)=0$. But each polynomial $g_{i}(z)=f_{i}(a,\eta -\lambda z,z)$ can only
have finitely many roots. This proves the claim.
\end{proof}
\end{remark}

This Remark tells us that the problem of determining (and computing) the
limit of a function along the varieties $X$ and $Y$ is equivalent to the
same problem when one approaches the origin along $X_{0}$ and $Y_{0}$.

\subsection{Groebner bases\label{Groebner}}

In this section we collect some basic properties and results on Groebner
bases and Elimination Theory that will be needed later for the development
of an algorithm that computes (\ref{xxx}). The main reference for this
section is \cite{Eisenbud}, Chapter 15.

By $S=K[x_{1},\dots ,x_{n}]$ we denote the polynomial ring in $n$-variables
with coefficients in a field $K$. We denote the set of monomials of $S$ by $%
M $. By a \emph{term} in $S$ is meant a polynomial of the form $cm$, where $%
c\neq 0\in K$ and $m\in M.$

\begin{definition}
A monomial order in $S$ is a total order on $M$ satisfying $%
nm_{1}>nm_{2}>m_{2},$ for every monomial $n\neq 1,$ and for any pair of
monomials $m_{1}$ and $m_{2}$ satisfying $m_{1}>m_{2}$.
\end{definition}

Every monomial order is Artinian which means that every subset of $M$ has a
least element.

For a fixed monomial order $>$ in $S$, the \emph{initial term} of $p\in S$
is the term of $p$ whose monomial is the greatest with respect to $>$. It is
usually denoted by $\text{in}(p)$. Given an ideal $I\subset S,$ its ideal of
initial terms, $\text{in}(I),$ is defined as the ideal generated by the set $%
\{\text{in}(p):p\in I\}$.

\begin{definition}
Let $I\subset S$ be any ideal, and fix a monomial order in $S$. We say that
a set of elements ~$\{f_{1},\dots ,f_{k}\}$ of $I$ is a \textit{Groebner
basis} for $I$ iff $\text{in}(I)=(\text{in}(f_{1}),\dots ,\text{in}(f_{k}))$.
\end{definition}

We list some basic facts about Groebner bases.

\begin{remark}
\label{rl1}

\begin{itemize}
\item[1.] The set of monomials not in the ideal $\text{in}(I)$ forms a basis
for the $K$-vector space $S/I$.

\item[2.] There always exists a Groebner basis for an ideal $I\subset S$.
As~ $S$ is a Noetherian ring, the ideal $I$ is finitely generated, let's
say, $I=(f_{1},\dots ,f_{k})$. Consider the ideal $J=(\text{in}(f_{1}),\dots
,\text{in}(f_{k}))$. If $J=\text{in}(I)$ then $\{f_{1},\dots ,f_{k}\}$ is a
Groebner basis for $I$.

\item[3.] If $\{f_{1},\dots ,f_{k}\}$ is a Groebner basis for $I$ then $%
I=(f_{1},\dots ,f_{k})$.

\item[4.] There is a criterion that allows to compute algorithmically a
Groebner basis for an ideal $I\subset S$. This criterion is known as
Buchberger's algorithm (\cite{Eisenbud}, Page 332).

\item[5.] Let $I,J$ be ideals of $S$ such that $I\subset J$. If $\text{in}%
(I)=\text{in}(J)$ then $I=J$.
\end{itemize}
\end{remark}

An example of a monomial order is the\emph{\ lexicographic order}, defined
in the following way: Fix any total order for the variables, for instance $%
x_{1}>x_{2}>\cdots >x_{n}$, and define $x_{1}^{a_{1}}x_{2}^{a_{2}}\cdots
x_{n}^{a_{n}}>x_{1}^{b_{1}}x_{2}^{b_{2}}\cdots x_{n}^{b_{n}}$ if for the
first $j$ with $a_{j}\neq b_{j}$ one has $a_{j}>b_{j}$. (The lexicographic
order will be the monomial order that we will use in this article.)

Now we discuss a basic result that will be needed in Sections 4 and 5.

Let $I$ be an ideal of the polynomial ring $K[x_{1},\dots x_{n},y_{1},\dots
,y_{s}]$. Given a Groebner basis for $I$ we want to compute a Groebner basis
for $I\cap K[x_{1},\dots ,x_{n}]$. For this, we have to introduce the notion
of an \textit{elimination order:}

\begin{definition}
A monomial order in $K[x_{1}, \dots x_{n}, y_{1}, \dots, y_{s}]$ is called
an elimination order if the following condition holds: $f \in K[x_{1}, \dots
x_{n}, y_{1}, \dots, y_{s}]$ with $\text{in}(f) \in K[x_{1}, \dots, x_{n}]$
implies $f \in K[x_{1}, \dots, x_{n}]$.
\end{definition}

\begin{lemma}
\label{ll1} Let $I\subset K[x_{1},\dots x_{n},y_{1},\dots ,y_{s}]$ be an
ideal, and let $\mathcal{B}=\{f_{1},\dots ,f_{k}\}$ be a Groebner basis for $%
I$ with respect to an elimination order. Assume that $f_{1},\dots ,f_{t}$
with $t\leq k$ are all elements of $\mathcal{B}$ such that $f_{1},\dots
,f_{t}\in K[x_{1},\dots ,x_{n}]$. Then $\{f_{1},\dots ,f_{t}\}$ is a
Groebner basis for $I\cap K[x_{1},\dots ,x_{n}]$.
\end{lemma}

\begin{proof}

(See \cite{Eisenbud}, Page 380).
\end{proof}

\begin{remark}
\label{rl2} Suppose that $\varphi :K[x_{1},\dots ,x_{n}]\rightarrow
K[y_{1},\dots ,y_{s}]/J$ is a ring homomorphism defined as $\varphi
(x_{i})=f_{i} $. Consider $F_{i}\in K[y_{1},\dots ,y_{s}]$ such that $%
\overline{F_{i}}=f_{i}$ in $K[y_{1},\dots ,y_{s}]/J,$ and define the ideal $%
I=JT+(F_{1}-x_{1},\dots ,F_{n}-x_{n})\subset T$, where $T=K[x_{1},\dots
,x_{n},y_{1},\dots ,y_{s}]$. Then $\ker {\varphi }=I\cap K[x_{1},\dots
,x_{n}]$. Therefore the above lemma implies that $\ker {\varphi }$ can be
computed algorithmically.
\end{remark}

\begin{proof}
	
(See \cite{Eisenbud}, Page 358).
\end{proof}

\section{Reduction to the case of functions of two variables}

Let $q(x,y,z)=f(x,y,z)/g(x,y,z)$ be the quotient of two polynomials. We
recall (Section \ref{discrimina}) that the \emph{discriminant variety
associated to }$q,$ $X(q)\subset \mathbb{C}^{3}$, is the affine variety
defined by the $2\times 2$ minors of the matrix we denoted by $A$. As a
variety, $X(q)$ may be decomposed into its irreducible components in $%
\mathbb{C}^{3},$ let's say $X(q)=X_{1}\cup X_{2}\cup \cdots \cup X_{n}$.

We are only interested in those components that contain the origin. These
will be called the \emph{relevant components}. Suppose these are $%
X_{1},X_{2},\dots ,X_{k},$ $k\leq n.$ We consider three possible cases:

\begin{itemize}
\item[1.] $\text{dim }X_{i}=0$: In this case, if $X_{i}=V(P_{i}),$ then $%
\mathbb{R}[x,y,z]/P_{i}$ is a field and $X_{i}$ is just the origin $\{O\}$.
Hence, $X_{i}$ does not contribute to any trajectory in $\mathbb{R}^{3}$
that approaches $O$, and can be discarded.

\item[2.] $\text{dim }X_{i}=1$: In this case $X_{i}$ is an irreducible
algebraic curve.

\item[3.] $\text{dim }X_{i}=2$: In this case $X_{i}$ is an hypersurface,
i.e., $X_{i}=V(P_{i}),$ where $P_{i}$ is a principal ideal.
\end{itemize}

We only have to study Cases $2$ and $3$.

We deal first with the case of an irreducible space curve in $\mathbb{C}^{3}$%
. Let us see that the problem of determining the limit of $q(x,y,z)$ along $%
X $, as well as its computation can be reduced to the case of a real plane
curve, a question already addressed in \cite{CSV}.\newline
By Theorem \ref{t2},~there is a plane curve $Y$ which is birationally
equivalent to $X$, and therefore a local isomorphism $\mu :X_{0}\rightarrow
Y_{0}$, where $X_{0}$ and $Y_{0}$ are as in Theorem \ref{t2}. There we
observed that the existence of the limit of $q(x,y,z)$ as $%
(x,y,z)\rightarrow (0,0,0)$ along $X_{0}$ is equivalent to the existence of
the limit of $q\circ \mu ^{-1}$ as $(u,v)\rightarrow (0,0)$ along $Y_{0}$.
Thus, 
\begin{equation}
\lim_{%
\begin{array}{c}
(x,y,z)\rightarrow O \\ 
(x,y,z)\in X_{0}%
\end{array}%
}q(x,y,z)~~=~~\lim_{%
\begin{array}{c}
(u,v)\rightarrow O \\ 
(u,v)\in Y_{0}%
\end{array}%
}q\circ \mu ^{-1}(u,v).  \label{cambio}
\end{equation}%
Summarizing:

\begin{proposition}
\label{pl3} Let $X\subset \mathbb{C}^{3}$ be an irreducible component of $%
X(q)$ of dimension $1$ containing $O$. Let $\mu :X_{0}\rightarrow Y_{0}$ be
the local isomorphism defined in Theorem \ref{t2}. Then, the limit of $%
q(x,y,z)$ as $(x,y,z)\rightarrow O$ along $X$ exists if and only if exists
along the irreducible plane curve $Y$ as $(u,v)\rightarrow (0,0).$ The
corresponding limits are related by (\ref{cambio}).
\end{proposition}

In the sequel we will denote by $X_{i}$ and $Y_{i}$ the two open subsets $%
X_{0}\subset X$ and $Y_{0}\subset Y$ defined in Theorem \ref{t2}. For the
purpose of analyzing the limit along the space curve $X$ it is only
necessary to consider those cases where the real trace of the birationally
isomorphic curve $Y$ turns out to be a\emph{\ plane curve containing the
origin.} By $\mu _{X_{i}}$ we denote the corresponding isomorphism between $%
X_{i}$ and $Y_{i}$ already constructed.

Now we analyze Case $3$. This is a lot more subtle, and requires a careful
analysis of the topology of the corresponding two dimensional component. A
key ingredient is a celebrated theorem of Whitney \cite{Whitney} about the
number of connected components of an affine algebraic variety. In the
following discussion we will show how one can reduce the analysis of the
2-dimensional irreducible components to Case 2.

Suppose that we have a rational function $q(x,y,z)=f(x,y,z)/g(x,y,z)$
defined on an irreducible hypersurface $X=V(h),$ where $h$ is a real
polynomial function of three variables and $q$ has an isolated zero at $0$.
Let $\mathcal{S}=\text{Sing}(X)$ be the singular locus of $X$. By Remark \ref%
{rsp}, $\mathcal{S}$ must be a variety of dimension strictly less than two.
Hence, if $\mathcal{S}$ contains the origin, the limit of $q$ as $%
(x,y,z)\rightarrow O$ along $\mathcal{S}$ can be computed as in Case 2.%
\newline
Now, we restrict our analysis to the nonsingular locus of $X$, that we
denote by $\mathcal{N}=X\setminus \mathcal{S}$. Without loss of generality
we may assume that $\mathcal{N}$ contains the origin, otherwise all of their
components would be irrelevant.\newline
Assume $O\in \mathcal{N}$, and define a family of real ellipsoids $%
E_{r}=\{(x,y,z)\in \mathbb{R}^{3}~:~Ax^{2}+By^{2}+Cz^{2}-r^{2}=0\}$, $%
A,B,C>0 $, $r\neq 0$. By $p_{r}(x,y,z)$ we will denote the quadratic
polynomial $Ax^{2}+By^{2}+Cz^{2}-r^{2}$.

\begin{definition}
Let $X=V(h) \subset \mathbb{C}^{3}$ and $E_{r}=\{(x,y,z) \in \mathbb{R}%
^{3}~:~Ax^{2}+By^{2}+Cz^{2}-r^{2}=0\}$, $r\neq 0$ as above. The critical set 
$C_{r}(q)$ will be the set of all real points in $E_{r} \cap X$ where $%
q(x,y,z)$ attains its maxima and minima. The union $\cup _{r>0}C_{r}(q)$ of
all critical sets will be denote by $\mathrm{Crit}_{X}(q)$.
\end{definition}

Since each $E_{r}\cap X$ is a compact set, and by hypothesis $O$ is an
isolated zero of $q$, the set $\mathrm{Crit}_{X}(q)$ is a well defined
subset of $X$.

We need the following analogue of Proposition \ref{pl1}.

\begin{proposition}
The limit $\lim_{(x,y,z)\rightarrow O}q(x,y,z)$ along $X$ exists and equals $%
L$ if and only if for every $\epsilon >0$ there is $\delta >0$ such that for
every $0<r<\delta $ the inequality $|q(x,y,z)-L|<\epsilon $ holds for all $%
(x,y,z)\in C_{r}$.
\end{proposition}

\begin{proof}
The proof follows identical lines as in Proposition \ref{pl1}. One just have to notice that each point in the critical set must lie in some $E_{r}$, since $p=(a,b,c)$
is obviously contained in $E_{r}$, with $r=\sqrt{Aa^{2}+Bb^{2}+Cc^{2}}$.
\end{proof}

Our objective is to determine $\text{Crit}_{X}(q)$. We can decompose this
set as the union of $\text{Crit}_{\mathcal{N}}(q)=\text{Crit}_{X}(q)\cap 
\mathcal{N}$ and $\text{Crit}_{X}(q)\cap \mathcal{S}$. Since $\text{Crit}%
_{X}(q)\cap \mathcal{S}\subset \mathcal{S},$ and the limit along $\mathcal{S}
$ can be determined as in Case 2, we just have to focus on $\text{Crit}_{%
\mathcal{N}}(q)$.\newline
First, we want to determine the nonsingular part of $\text{Crit}_{\mathcal{N}%
}(q)$ by using the method of Lagrange Multipliers, as in \cite{CSV}. For
this we define $\mathfrak{X}=V(\mathfrak{J})\subset X$ to be the zero set of
the ideal $\mathfrak{J}$ generated by $h$ and the determinant: 
\begin{equation*}
d(x,y,z)=\left\vert 
\begin{array}{ccc}
\partial p_{r}/\partial x & \partial p_{r}/\partial y & \partial
p_{r}/\partial z \\ 
\partial h/\partial x & \partial h/\partial y & \partial h/\partial z \\ 
\partial q/\partial x & \partial q/\partial y & \partial q/\partial z%
\end{array}%
\right\vert .
\end{equation*}%
As the points of $X$ already satisfy $\nabla q(x,y,z)=\lambda (x,y,z)$
(where $\nabla q$ denotes the gradient of $q$), and since $\nabla
p_{r}(x,y,z)=(2Ax,2By,2Cz)$, the affine variety $\mathfrak{X}$ must be
defined by the ideal generated by $h$ and by the determinant: 
\begin{equation*}
D(x,y,z)=\left\vert 
\begin{array}{ccc}
Ax & By & Cz \\ 
x & y & z \\ 
\partial h/\partial x & \partial h/\partial y & \partial h/\partial z%
\end{array}%
\right\vert .
\end{equation*}%
That is, $\mathfrak{X}=V(D,h)$. This variety is precisely the set of regular
points of $X$ that are critical points of $q$.

\begin{proposition}
\label{nueva}(Notation as above) Let us assume $O\in \mathcal{N}$. Then it
is possible to choose (in a generic way) suitable positive constants $A,B$
and $C$ such that the height of the ideal $\mathfrak{J}=(D,h)$ in the
polynomial ring $\mathbb{C}[x,y,z]$ is greater than one, and consequently $%
\dim \mathfrak{X}<2$.
\end{proposition}

\begin{proof}
It suffices to show that for a suitable choice of positive constants $A,B,C$
there is at least one point $p\neq O$ in $\mathcal{N}$ such that $D(p)\neq 0$%
.

First, let us see that there is at least one point $p\in \mathcal{N}$
different from the origin such that the gradient of $h$ does not point in
the direction of $p$, i.e., such that $\nabla h(p)\neq \lambda p,$ for all $%
\lambda \in \mathbb{R}$. Indeed, suppose on the contrary that for every $%
p\in \mathcal{N}$ there existed $\lambda (p)\neq 0$ such that $\nabla
h(p)=\lambda (p)p$. Since each $p$ is a regular point of $X$, one must have $%
\nabla h(p)\neq 0.$ Hence, after making an appropriated change of
coordinates that fixes $O$ (a rotation, and then a homothety) we may assume
without loss of generality that $\partial h/\partial z(0,0,1)\neq 0$, and
that $p=(0,0,1)$. By the implicit function theorem there would exist $%
U_{0}\subset \mathbb{R}^{2},$ a neighborhood of $(0,0),$ and a smooth function $%
u(x,y)$ in $U_{0}$ such that $u(0,0)=1,$ and $h(x,y,u(x,y))=0,$ for all $%
(x,y)\in U_{0}$. Since $\nabla h(p)=\lambda (p)p$, one must have $\partial
h/\partial x(0,0,1)=\partial h/\partial y(0,0,1)=0,$ and consequently $%
\partial u/\partial x(0,0)=0=\partial u/\partial y(0,0)$.

Let $W_{p}$ be the graph $W_{p}=\{(x,y,u(x,y))~:~(x,y)\in U_{0}\}$. For any $%
\mathfrak{t}\in W_{p}$, the normal vector at $\mathfrak{t}$ is given by 
\begin{equation*}
n(\mathfrak{t})=\frac{(-u_{x},-u_{y},1)}{\sqrt{u_{x}^{2}+u_{y}^{2}+1}}.
\end{equation*}%
Henceforth, if $\mu (\mathfrak{t})=\lambda (\mathfrak{t})/\Vert \nabla h(%
\mathfrak{t})\Vert $ one has that $\nabla h(\mathfrak{t})=\mu (\mathfrak{t)}%
\Vert \nabla h(\mathfrak{t})\Vert \mathfrak{t}$, and consequently $n(%
\mathfrak{t})$ can be written as 
\begin{equation*}
n(\mathfrak{t})=\frac{(x,y,u(x,y))}{\sqrt{x^{2}+y^{2}+u^{2}(x,y)}}.
\end{equation*}%
From this, we deduce:~ 
\begin{eqnarray*}
\frac{1}{\sqrt{u_{x}^{2}+u_{y}^{2}+1}} &=&\frac{u(x,y)}{\sqrt{%
x^{2}+y^{2}+u^{2}(x,y)}},~ \\
\frac{-u_{x}}{\sqrt{u_{x}^{2}+u_{y}^{2}+1}} &=&\frac{x}{\sqrt{%
x^{2}+y^{2}+u^{2}(x,y)}},
\end{eqnarray*}%
and 
\begin{equation*}
\frac{-u_{y}}{\sqrt{u_{x}^{2}+u_{y}^{2}+1}}=\frac{y}{\sqrt{%
x^{2}+y^{2}+u^{2}(x,y)}}.
\end{equation*}%
This implies~ $u_{x}=-x/u(x,y),$ and $u_{y}=-y/u(x,y)$. Hence, $u(x,y)=\sqrt{%
1-x^{2}-y^{2}}$, since $u(0,0)=1$. \emph{We conclude that }$W_{p}$\emph{\
would be a neighborhood of }$p$ \emph{in} $\mathcal{N}$\emph{\ which is part
of a sphere centered at the origin. } But\emph{\ }on the other hand, a
theorem of Whitney asserts that $\mathcal{N}$ can only have finitely many
connected components (see \cite{Whitney}). Then this would  imply that $%
\mathcal{N}$ could not contain the origin, a contradiction with our assumption.

Therefore, we may assume there exists a point $p\neq O$ in $\mathcal{N}$
such that $\nabla h(p)\neq \lambda p,$ for all $\lambda \neq 0$. After
applying a rotation (if necessary) we may also assume that $a,b,c$ are all
nonzero.

After those preliminaries it becomes clear how to choose positive constants $A,B$
and $C$ such that the determinant 
\begin{equation*}
\left\vert 
\begin{array}{ccc}
Aa & Bb & Cc \\ 
a & b & c \\ 
\partial h/\partial x(a,b,c) & \partial h/\partial y(a,b,c) & \partial
h/\partial z(a,b,c)%
\end{array}%
\right\vert 
\end{equation*}%
does not vanish: The vectors $\nabla h(p)$ and $p=(a,b,c)$ generate a plane $%
H,$ since they are not parallel. Therefore, it suffices to choose any point $%
(\alpha ,\beta ,\gamma )$ outside $H$ and such that $A=\alpha /a$, $B=\beta
/b$, and $C=c/\gamma $ are positive.
\end{proof}

As before, for the limit $\lim_{(x,y,z)\rightarrow O}q(x,y,z)$ to exist
along $X$ it is necessary that it exists along any real curve that contains $%
O $. In particular, the limit along each component of $\mathfrak{X}$ must
exist, and all theses limits must be equal. By Proposition 17, $\text{dim}(%
\mathfrak{X})<2$, and henceforth we can reduce this last question to cases 1
and 2.\newline
Let $\mathfrak{Z}$ be the affine variety defined by the ideal generated by $%
h $ and by the minors $2\times 2$ of the matrix 
\begin{equation*}
\left[ 
\begin{array}{ccc}
Ax & By & Cz \\ 
\partial h/\partial x & \partial h/\partial y & \partial h/\partial z%
\end{array}%
\right] .
\end{equation*}%
The set $\mathfrak{Z}\cap E_{r}\cap \mathcal{N}$ defines the locus of those
real points where $E_{r}$ and $\mathcal{N}$ do not intersect transversely.
Outside this set, $E_{r}\cap \mathcal{N}$ is a $1$-dimensional manifold (see 
\cite{GP}, Page 30) that we shall denote by $\Sigma $. Clearly, the
vanishing of these two by two minors forces the vanishing of the determinant 
$D(x,y,z)$. Henceforth, $\mathfrak{Z}\subset \mathfrak{X}$, and consequently 
$\text{dim}(\mathfrak{Z})<2$, by Proposition \ref{nueva}.\newline
Again, for the existence of the limit $\lim_{(x,y,z)\rightarrow O}q(x,y,z)$
it is required, in particular, its existence along any relevant component of 
$\mathfrak{Z}$, and consequently the problem reduces again to cases 1 and 2.
This takes care of the subset of $\text{Crit}_{\mathcal{N}}(q)$ inside $%
\mathfrak{Z}$.\newline
As for those points in $\text{Crit}_{\mathcal{N}}(q)$ that lie outside $%
\mathfrak{Z}$, we notice that they are contained in the $1$-dimensional
manifold $\Sigma $. Then they must be part of $\mathfrak{X}$, since this
variety is precisely those regular points where $q$ attains an extreme
value. Thus, the points in $\text{Crit}_{\mathcal{N}}(q)$ that lie outside $%
\mathfrak{Z}$ must be contained in $\mathfrak{X}$. Once again, we have
reduced the problem to cases 1 and 2.\newline

The following proposition summarizes this discussion:

\begin{proposition}
\label{pl4} Let $X$ be a relevant irreducible component of dimension $2$ of
the discriminant variety $X(q)$. Consider $\mathcal{S}$, $\mathfrak{X}$, and 
$\mathfrak{Z}$ as defined above. Then, the limit of $q(x,y,z)$ as $%
(x,y,z)\rightarrow O$ along $X$ exists, and equals $L$, if and only if, the
limit of $q(x,y,z)$ as $(x,y,z)\rightarrow O$ exists and equals $L$ along
each one of the components of the curves $\mathcal{S}$, $\mathfrak{X}$, and $%
\mathfrak{Z}$.
\end{proposition}

We are now ready to state our main result.

\begin{theorem}
\label{tl2} Let $q(x,y,z)=f(x,y,z)/g(x,y,z),$ where $f$ and $g$ are rational
polynomial functions, and where $g$ has an isolated zero at the origin. Let $%
X(q)$ be the discriminant variety associated to $q$. Denote by $%
\{X_{1},\dots ,X_{k}\}$ the relevant irreducible components of dimension one
of $X(q)$, and by $\{X_{k+1},\dots ,X_{n}\}$ the relevant irreducible
components of dimension two of $X(q)$. Then, the limit of $q$ as $%
(x,y,z)\rightarrow O$ exists, and equals $L,$ if and only if the limit of $%
q(x,y,z)$ as $(x,y,z)\rightarrow O$ along $X_{i}$ exists, and equals $L$,
for all $i=1,2,\dots ,n$. Moreover:

\begin{enumerate}
\item For the components $X_{i}$, $i=1,2,\dots ,k,$ the limit of $q(x,y,z)$
as $(x,y,z)\rightarrow (0,0,0)$ along $X_{i}$ is determined as in
Proposition \ref{pl3}.

\item For the components $X_{j},$ $j=k+1,\dots ,n$, the limit of $q(x,y,z)$
as $(x,y,z)\rightarrow (0,0,0)$ along $X_{j}$ is determined as in
Proposition \ref{pl4}.
\end{enumerate}
\end{theorem}

\section{A high level description of an algorithm for computing the limit 
\label{HLA}}

Let $q(x,y,z)=f(x,y,z)/g(x,y,z),$ where $f$ and $g$ are polynomial functions
of three variables with rational coefficients, and $g$ has an isolated zero
at the origin. Consider $X(q),$ the discriminant variety associated to $q$.
We have to decompose $X(q)$ into irreducible components, and then choose
only those irreducible components $\{X_{1},\dots ,X_{n}\}$ that are relevant.%
\newline

The algorithm has to deal with two different cases:

\begin{itemize}
\item \textbf{D1}: The component $X_{i}$ has dimension $1$. Then as observed
before, $X_{i}$ is birationally equivalent to an irreducible plane curve $%
Y_{i}$. Let us denote by $\mathbb{C}(x,y,z)$ the fraction field of the ring
of coordinates $\mathbb{C}[X,Y,Z]/I(X_{i})$ of $X_{i}$. As we already
noticed we may always assume that $x,y,z$ are transcendental elements over $%
\mathbb{C}$: If, for instance, $x$ were algebraic over $\mathbb{C}$, then
there would exist a polynomial $P(X)\in \mathbb{C}[X]$ such that $P(x)=0.$
This is equivalent to saying that $P(X)\in I(X_{i})$. Suppose we write $%
P(X)=(X-\alpha _{1})(X-\alpha _{2})\cdots (X-\alpha _{n})$ in $\mathbb{C}[X]$%
. Since $I(X_{i})$ is a prime ideal, some linear factor $X-\alpha _{j}$ must
belong to $I(X_{i}).$ But as $X_{i}$ contains the origin, we must have $%
\alpha _{j}=0$. Hence, we could write $I(X_{i})=(X,h_{1}(Y,Z),\dots
,h_{m}(Y,Z)),$ where $h_{k}(Y,Z)\in \mathbb{C}[Y,Z],$ for $k=1,2,\dots ,m$.
If we denote by $I$ the ideal $(h_{1}(Y,Z),\dots ,h_{m}(Y,Z)),$ and by $%
X_{i}^{\prime }=V(I)\subset \mathbb{C}^{2}$ the affine variety defined by $I$%
, then the limit of $q(X,Y,Z)$ as $(X,Y,Z)\rightarrow O$ along $X_{i}$ is
the same as the limit of $q(0,Y,Z)$ as $(Y,Z)\rightarrow (0,0)$ along $%
X_{i}^{\prime }$. But the existence of this limit, as well as its value, can
be computed using the algorithmic method developed in \cite{CSV}. By
Proposition \ref{t1} and Corollary \ref{c1} we know that if $x$ is
transcendental over $\mathbb{C}$ there exists $\lambda \in \mathbb{C}(x)$
such that $\mathbb{C}(x,y,z)=\mathbb{C}(x,u),$ where $u=y+\lambda z$. Also,
by Theorem \ref{t2}, if we consider $\varphi :\mathbb{C}[X,U]\rightarrow 
\mathbb{C}[X,Y,Z]/I(X)$ defined by $\varphi (X)=x$, $\varphi (U)=y+\lambda z$%
, then $\ker {\varphi }$ defines the irreducible plane curve $Y$ that is
birationally equivalent to $X$. As we observed in Section \ref{Groebner}, $%
\ker {\varphi }=(I(W)T+(U-(Y+\lambda Z)))\cap \mathbb{C}[X,U],$ where $T=%
\mathbb{C}[X,U,Y,Z]$ is computable. On the other hand, the ring homomorphism 
$\mathbb{C}[X,U]/\ker {\varphi }\rightarrow \mathbb{C}[X,Y,Z]/I(X)$ induces
an isomorphism of fields $\varphi :K(X)\rightarrow K(Y)$. As we showed in
the proof of Theorem \ref{t2}, since $y,z\in \mathbb{C}(x,u)$ then one must
have $y=f_{1}(x,u)/g_{1}(x),$ and $z=f_{2}(x,u)/g_{2}(x),$ for some $%
f_{1},f_{2},g_{1},$ and $g_{2}$ with real coefficients. In that same proof
we noticed that the local isomorphism $\mu :X_{0}\rightarrow Y_{0}$ is
determined by those polynomials.

By Proposition \ref{pl3}, computing the limit of $q(X,Y,Z)$ as $%
(X,Y,Z)\rightarrow O$ along $X_{i}$ is equivalent to computing the limit of $%
q\circ \mu _{X_{i}}^{-1}(X,U),$ as $(X,U)\rightarrow (0,0)$ along $Y_{i}$,
and this last limit can be dealt with using the algorithm developed in \cite%
{CSV}.

\item \textbf{D2:} Suppose that $\text{dim}X_{i}=2$. Then $X_{i}$ is an
affine variety defined by a principal ideal $I(X_{i})=(h)$. For random
positive values $A$, $B$ and $C$ the algorithm computes the height of the
ideal $J=(D,h),$ where 
\begin{equation*}
D=\left\vert 
\begin{array}{ccc}
Ax & By & Cz \\ 
x & y & z \\ 
\partial h/\partial x & \partial h/\partial y & \partial h/\partial z%
\end{array}%
\right\vert .
\end{equation*}%
As we saw in the reduction to plane curves, there always exist positive
constants $A,B$ and $C$ such that $\text{ht}(\mathfrak{J})\geq 2$. Since $%
\text{dim}(\mathfrak{X})\leq 1$, then $\mathfrak{X}=V(\mathfrak{J})$.
Henceforth, one can compute the limit of $q(x,y,z)$ as $(x,y,z)\rightarrow O$
along $\mathfrak{X}$ using the prescription in \textbf{D1}.\newline
Since $\mathcal{S}=\text{Sing(}X)$, the affine variety defined by the ideal $%
(h,\frac{\partial h}{\partial x},\frac{\partial h}{\partial y},\frac{%
\partial h}{\partial z})$ must be a proper subset of $X$. Then $\mathcal{S}$
is also an algebraic curve, and once again we can compute the limit of $%
q(x,y,z)$ as $(x,y,z)\rightarrow O$ along $\mathcal{S}$ using \textbf{D1}.%
\newline
Now, the affine variety $\mathfrak{Z}$ defined by the ideal generated by the
minors $2\times 2$ of the matrix 
\begin{equation*}
\left[ 
\begin{array}{ccc}
Ax & By & Cz \\ 
\frac{\partial h}{\partial x} & \frac{\partial h}{\partial y} & \frac{%
\partial h}{\partial z}%
\end{array}%
\right]
\end{equation*}%
and the polynomial $h$, has also dimension less than $2$. Hence, the limit
of $q(x,y,z)$ as $(x,y,z)\rightarrow O$ along $\mathfrak{Z}$ is also
computed using \textbf{D1}.

\item Finally, if the limit of $q(x,y,z)$ as $(x,y,z)\rightarrow O$ along
each relevant irreducible component of $X(q)$ of dimension one exists, and
equals $L$, one says that the limit of $q(x,y,z)$ as $(x,y,z)\rightarrow O$
is $L$. Otherwise, one says that this limit does not exists.
\end{itemize}

\section{Examples}

\subsection{Example $1$}

Suppose that we want to compute the limit: 
\begin{equation*}
\lim_{(X,Y,Z)\rightarrow (0,0,0)}\frac{YX-ZY+ZX}{X^{2}+Y^{2}+Z^{2}}.
\end{equation*}%
Let $q(X,Y,Z)=YX-ZY+XZ/X^{2}+Y^{2}+Z^{2}$. We illustrate the necessary
computations, carried out in the program \textit{Maple}.

\begin{enumerate}
\item Using the command \texttt{PrimeDecomposition(X(q))} one gets the
irreducible components of $X(q)$: 
\begin{eqnarray*}
&&V((Y-X+Z)),~~V((X^{2}+Y^{2}+Z^{2})),~~V((X+Y,Z-2X)),~~ \\
&&V((X+Y,Z+X))~~\text{and}~~V((X+Y,Z^{2}+2X^{2})).
\end{eqnarray*}%
By using the command \texttt{HilbertDimension(Q)} one can see that the
irreducible component $V(X+Y,Z+X)$ has dimension $1$.

Let us see that for $\lambda =1$, $\mathbb{C}(x,u)=\mathbb{C}(x,y,z),$ where 
$u=y+z$. Here, $\mathbb{C}(x,y,z)$ denotes the fraction field of the ring of
coordinates of the variety $V(X+Y,Z+X)$. Consider the ideal $%
I=(X+Y,Z+X)T+(U-(Y+Z)),$ where $T=\mathbb{C}[X,Y,Z,U]$. The command \texttt{%
EliminationIdeal(I,\{U,X\})} generates the ideal $J=(2X+U)$. On the other
hand, the command \texttt{Basis(I,plex(Z,Y,U))} gives us a Groebner basis
for $I$ respect to the lexicographic monomial order, with $Z>Y>U$. In this
particular case we obtain the following basis: $\{2X+U,Y+X,Z+X\}$. From this
basis we deduce that $y=-x$ and $z=-x$ are elements of $\mathbb{C}(x,u)$.
Therefore, $\mathbb{C}(x,u)=\mathbb{C}(x,y,z),$ and consequently the ideal $%
J=(2X+U)$ defines an irreducible plane curve which is birationally
equivalent to $V(X+Y,Z+X)$. Also, $y=-x$ and $z=-x$ determine the
isomorphism $\rho :V(2X+U)\rightarrow V(X+Y,Z+X)$. Therefore, the limit of $%
q(X,Y,Z),$ as $(X,Y,Z)\rightarrow (0,0,0)$ along $V(X+Y,Z+X),$ is equivalent
to the limit of $q\circ \rho (X,U)$ as $(X,U)\rightarrow (0,0)$ along $%
V(2X+U)$. This latter limit can be computed using the algorithm developed in 
\cite{CSV}. However, in this case it is easy to see directly that the value
of the limit is $-1/3,$ since $q\circ \rho (X,U)=q(X,-X,-X)=-1/3$.
Therefore, the limit of $q(X,Y,Z)$ as $(X,Y,Z)\rightarrow (0,0,0)$ along $%
V(X+Y,Z+X)$ is $-1/3$.\newline
Let $h(X,Y,Z)=Y-X+Z.$ One may choose $A=1$, $B=2,$ and $C=1$. Using the
command \texttt{HilbertDimension(P)}\textbf{,} with $P=(f,h)$, where 
\begin{equation*}
f=\left\vert 
\begin{array}{ccc}
X & 2Y & Z \\ 
X & Y & Z \\ 
\frac{\partial h}{\partial X} & \frac{\partial h}{\partial Y} & \frac{%
\partial h}{\partial Z}%
\end{array}%
\right\vert ,
\end{equation*}%
one obtains that the variety defined by the ideal $P=(f,h)$ has dimension $1$%
. In this case $V(P)=V(-XY-YZ,Y-X+Z)$. Using again the command \texttt{%
PrimeDecomposition(P)} one obtains the irreducible components of the variety 
$V(P)$: $V(P)=V(Y,-X+Z)\cup V(2X-Y,Y+2Z),$ where each of these components
has dimension $1$. Therefore, one just needs to compute a limit along
irreducible algebraic curves (again, using the main algorithm of \cite{CSV}%
). For the variety $V(Y+X,Z+X)$ we may follow an analogous procedure. It is
not difficult to see that the limit of $q(X,Y,Z),$ as $(X,Y,Z)\rightarrow
(0,0,0)$ along the variety $V(Y,-X+Z)$ is equal to $1/2$.\newline
Hence, we conclude that 
\begin{equation*}
\lim_{(X,Y,Z)\rightarrow (0,0,0)}\frac{YX-ZY+ZX}{X^{2}+Y^{2}+Z^{2}}
\end{equation*}%
does not exist.
\end{enumerate}

\subsection{Example $2$}

\begin{enumerate}
\item We want to compute the limit: 
\begin{equation*}
\lim_{(X,Y,Z)\rightarrow (0,0,0)}\frac{X^{2}YZ}{X^{2}+Y^{2}+Z^{2}}
\end{equation*}%
Let $q(X,Y,Z)=X^{2}YZ/X^{2}+Y^{2}+Z^{2}$.

Using the command \texttt{PrimeDecomposition(X(q)}\textit{\ }one obtains the
irreducible components of $X(q)$:

\item The irreducible components of dimension $1$ are: $X_{1}=V(X,Y)$,~ $%
X_{2}=V(X,Z)$,~ $X_{3}=V(Y,Z)$,~ $X_{4}=V(X,Z-Y)$,~ $X_{5}=V(X,Z+Y)$,~ $%
X_{6}=V(X,Y^{2}+Z^{2})$,~ $X_{7}=V(X,3Z^{2}-Y^{2})$,~ $%
X_{8}=V(X,3Z^{2}+Y^{2})$, ~$X_{9}=V(Z,X^{2}+Y^{2})$,~ $%
X_{10}=V(Z+Y,-2Z^{2}+X^{2})$,~ $X_{11}=V(-Z+Y,-2Z^{2}+X^{2})$,~ and~ $%
X_{12}=V(-2Z^{2}+X^{2},3Z^{2}+Y^{2})$.

\item The irreducible components of dimension $2$ are: $X_{13}=V(X),$~ and~ $%
X_{14}=V(X^{2}+Y^{2}+Z^{2})$.

\item There is an irreducible component of dimension $0$: $X_{15}=V(X,Y,Z)$.

We know that each irreducible component of dimension $1$ is birationally
equivalent to an irreducible plane curve. Now, if any of the variables $X,Y$
or $Z$ appears in the ideal that defines the corresponding irreducible
component, then one can easily see that such component in $\mathbb{C}^{3}$
is actually contained in $\mathbb{C}^{2}$. Henceforth, it is already a plane
curve, and one can use the main algorithm of \cite{CSV} to compute these
(two variable) limits. Hence, one sees that the limits along the varieties $%
X_{1}=V(X,Y)$,~ $X_{2}=V(X,Z)$,~ $X_{3}=V(Y,Z)$,~ $X_{4}=V(X,Z-Y)$,~ $%
X_{5}=V(X,Z+Y)$,~ $X_{6}=V(X,Y^{2}+Z^{2})$,~ $X_{7}=V(X,3Z^{2}-Y^{2})$,~ $%
X_{8}=V(X,3Z^{2}+Y^{2})$, and ~$X_{9}=V(Z,X^{2}+Y^{2})$ are equal to zero.

Now we discuss the limit along the other irreducible components of dimension 
$1$.

\item For $X_{10}=V(Z+Y,-2Z^{2}+X^{2})$: We noticed in the proof of the
Primitive Element Theorem that for almost all $\lambda \in \mathbb{R}$, $%
\mathbb{R}(x,y,z)=\mathbb{R}(x,u)$, where $u=y+\lambda z$. In this case, one
could take $\lambda =2$. Let $I=(Z+Y,-2Z^{2}+X^{2},U-(Y+2Z))\subset \mathbb{R%
}[X,Y,Z,U]$, with the command \texttt{EliminationIdeal}$(I,{X,U})$ one gets
the plane curve $V(2U^{2}-X^{2}),$ which is birationally equivalent to $%
X_{10}$. On the other hand, by using the command \texttt{Basis}$%
(I,plex(Z,Y,U))$ one computes the basis $\{2U^{2}-X^{2},Y+U,-U+Z\}$. From
this basis we deduce that $y=-u$ and $z=u$ as elements of $\mathbb{R}(x,u)$.
Thus, the limit along the component $X_{10}$ is the same as the limit of $%
q(X,-U,U)$ along the irreducible plane curve $V(2U^{2}-X^{2}).$ This latter
limit can be calculated using the the main algorithm of \cite{CSV}. In this
case we obtain the value zero.

\item For $X_{11}=V(-Z+Y,-2Z^{2}+X^{2})$, with $\lambda =1$ and following
the same procedure, i.e., defining the ideal $I=(-Z+Y,-2Z^{2}+X^{2},U-(Y+Z))$
and then computing \texttt{EliminationIdeal}$(I,{X,U})$ and \texttt{Basis}$%
(I,plex(Z,Y,U))$, one obtains the irreducible plane curve $V(U^{2}-2X^{2}),$
which is birationally equivalent to $X_{11},$ as well as the basis $%
\{U^{2}-2X^{2},-U+2Y,-U+2Z\}$. From this basis one deduces that $y=u/2$ and $%
z=u/2$, and therefore the limit along the component $X_{11}$ is the same as
the limit of $q(X,U/2,U/2)$ along the plane curve $V(U^{2}-2X^{2})$ which is
again a limit in two variables, and can be computed using the methods of 
\cite{CSV}. In this case the limit is also zero.

\item For $X_{12}=V(-2Z^{2}+X^{2},3Z^{2}+Y^{2})$, with $\lambda =1$, by
defining the ideal 
\begin{equation*}
I=(-2Z^{2}+X^{2},3Z^{2}+Y^{2},U-(Y+Z)),
\end{equation*}
and then computing \texttt{EliminationIdeal}$(I,{X,U})$ and \texttt{Basis}$%
(I,plex(Z,Y,U))$, one obtains the irreducible plane curve $%
V(4X^{4}+2X^{2}U^{2}+U^{4}),$ which is birationally equivalent to $X_{12},$
as well as the basis 
\begin{equation*}
\{4X^{4}+2X^{2}U^{2}+U^{4},-U^{3}-4UX^{2}+4X^{2}Y,4ZX^{2}+U^{3}\}.
\end{equation*}
From this we deduce $y=\frac{u^{3}+4ux^{2}}{4x^{2}},$ and $z=\frac{-u^{3}}{%
4x^{2}}$, and therefore the limit along the component $X_{12}$ is the same
as the limit of $q(X,\frac{U^{3}+4UX^{2}}{4X^{2}},\frac{-U^{3}}{4X^{2}})$
along the plane curve $V(4X^{4}+2X^{2}U^{2}+U^{4}),$ which is again a limit
in two variables. This limit is also zero.

\item Now, the components of dimension $2$ are $V(X^{2}+Y^{2}+Z^{2})$ and $%
V(X)$. The first component is precisely the set of points where the rational
function $q$ is not defined, and consequently can be discarded. Since the
variable $X$ appears in the ideal defining the second variety the limit
clearly must be zero.\newline
We conclude that 
\begin{equation*}
\lim_{(X,Y,Z)\rightarrow (0,0,0)}\frac{X^{2}YZ}{X^{2}+Y^{2}+Z^{2}}=0,
\end{equation*}%
since it is zero along each of the irreducible components of the
discriminant variety.
\end{enumerate}

\end{document}